\documentclass{amsart}

\usepackage{amsmath, amssymb,epic,graphicx,mathrsfs,enumerate}
\usepackage[all]{xy}

\usepackage{latexsym}
\usepackage{longtable}
\usepackage{epsfig}
\usepackage{hhline}

\newtheorem{proposition}{Proposition}[section]

\newtheorem{lemma}[proposition]{Lemma}
\newtheorem{corollary}[proposition]{Corollary}
\newtheorem{theorem}[proposition]{Theorem}

\begin{document}

\title{Growth of products of subsets in finite simple groups}

\author{Daniele Dona}
\address{Alfr\'ed R\'enyi Institute, Re\'altanoda Utca 13-15, H-1053, Budapest, Hungary}
\email{dona.daniele@renyi.hu}

\author{Attila Mar\'oti}
\address{Alfr\'ed R\'enyi Institute, Re\'altanoda Utca 13-15, H-1053, Budapest, Hungary}
\email{maroti.attila@renyi.hu}

\author{L\'aszl\'o Pyber}
\address{Alfr\'ed R\'enyi Institute, Re\'altanoda Utca 13-15, H-1053, Budapest, Hungary}
\email{pyber.laszlo@renyi.hu}

\thanks{The project leading to this application has received funding from the European Research Council (ERC) under the European Union's Horizon 2020 research and innovation programme (grant agreement No.~741420). The second and third authors were also supported by the National Research, Development and Innovation Office (NKFIH) Grant No.~K138596. The second author was also funded by No.~K132951 and Grant No.~K138828.}

\begin{abstract}
We prove that the product of a subset and a normal subset inside any finite simple non-abelian group $G$ grows rapidly. More precisely, if $A$ and $B$ are two subsets with $B$ normal and neither of them is too large inside $G$, then $|AB| \geq |A||B|^{1-\epsilon}$ where $\epsilon>0$ can be taken arbitrarily small. This is a somewhat surprising strengthening of a theorem of Liebeck, Schul, Shalev.

\end{abstract}
\maketitle

\section{Introduction}

The study of growth of products of subsets in finite simple groups has been the subject of significant work in the recent decades. Part of the interest revolves around a conjecture of Liebeck, Nikolov, and Shalev \cite{LNS}, which claims that for any finite simple non-abelian group $G$ and any set $A\subseteq G$ of size at least $2$ we can write $G$ as the product of $N$ conjugates of $A$ with $N=O(\log|G|/\log|A|)$. This conjecture generalizes an already deep theorem of Liebeck and Shalev \cite{LS}, which proves it for $A$ a \textit{normal} subset, i.e.\ a union of conjugacy classes of $G$.

In attempting to prove the conjecture, or partial cases thereof, a natural way is to show that the product of two subsets has size comparable to the product of the sizes of the two original sets. A result in this vein is the following, due to Gill, Pyber, Short, and Szab\'o \cite[Proposition 5.2]{GPSSz}. For any $\epsilon > 0$ there exists $\delta > 0$ such that if $G$ is a finite simple non-abelian group, $A$ is a subset with $|A|\leq|G|^{1-\delta}$, and $B$ is a normal subset, then $|AB| \geq |A| {|B|}^{\epsilon}$. This theorem strengthens the expansion result given in \cite[Proposition 10.4]{Shalev} for conjugacy classes that are not too large with respect to the size of $G$. Liebeck, Schul, and Shalev later used another result of this kind to prove that for small classes, and indeed for small normal subsets, the expansion is particularly rapid. They proved \cite[Theorem 1.3]{LSS} that for any $\epsilon > 0$ there exists $\delta > 0$ such that if $G$ is a finite simple non-abelian group and $A,B$ are two normal subsets with $|A|,|B|\leq|G|^{\delta}$, then $|AB| \geq (|A| |B|)^{1-\epsilon}$.

In the present paper we prove the following.

\begin{theorem}\label{main}
For any $\epsilon > 0$ there exists $\delta > 0$ such that if $G$ is a finite simple non-abelian group, $A$ is a subset and $B$ is a normal subset with $|A|,|B|\leq|G|^{\delta}$, then $|AB| \geq |A| |B|^{1-\epsilon}$.
\end{theorem}

Theorem~\ref{main} is a direct generalization of \cite[Theorem 1.3]{LSS}, and it improves \cite[Proposition 5.2]{GPSSz} for sets of size at most $|G|^{\delta}$.

\section{Bounding conjugacy class sizes in alternating groups}

In this section let $G$ be the alternating group of degree $r$ and let $x \in G$. We define $\Delta(x)$ to be $(r-t)/r$ where $t$ denotes the number of cycles in the disjoint cycle decomposition of $x$. The purpose of this section is to show that, unlike the support of $x$, the invariant $\Delta(x)$ controls the size of the conjugacy class $x^{G}$, provided that it is small.   

We will need a variant of \cite[Lemma 2.3]{GaronziMaroti}. 

\begin{lemma}
\label{GaronziMaroti}	
For every $\gamma$ and $\epsilon$ with $0 < \gamma < 1$ and $0 < \epsilon < 1$ there exists $N$ such that for every $r \geq N$, whenever $x \in G$ satisfies $|x^{G}| \geq |G|^{\gamma}$, then $\Delta(x) > (1-\epsilon)\gamma$.
\end{lemma}

\begin{proof}
Fix $\gamma$ and $\epsilon$ with $0 < \gamma < 1$ and $0 < \epsilon < 1$. According to \cite[Lemma 2.3]{GaronziMaroti}, for every $\epsilon_{1} > 0$ there exists $N_1$ such that for every $r \geq N_{1}$, whenever $x \in G$ satisfies $|x^{G}| \geq |G|^{\gamma}$, then $\Delta(x) > \gamma - \epsilon_{1}$. It is sufficient to choose $\epsilon_{1}$ such that $\gamma - \epsilon_{1} > (1 - \epsilon) \gamma$. This is the case when $\epsilon_{1} < \gamma \epsilon$. 
\end{proof}

We need the following bounds of Stirling found in \cite[2.9]{Feller}. 

\begin{lemma}
	\label{Stirling}	
	For every positive integer $n$ we have 
	$$\sqrt{2 \pi n} {\Big( \frac{n}{e} \Big)}^{n} \leq n! \leq 2 \sqrt{2 \pi n} {\Big( \frac{n}{e} \Big)}^{n}.$$ 
\end{lemma}

We are now in position to prove the main result of this section.

\begin{proposition}
\label{classsize1}	
For all $\epsilon > 0$ there exists $\delta > 0$ such that whenever $G$ is an alternating group and $x \in G$ with $|x^{G}| \leq |G|^{\delta}$, then 
$$|G|^{\Delta(x)(1- \epsilon)} \leq |x^{G}| \leq |G|^{\Delta(x)(1+ \epsilon)}.$$
\end{proposition}

\begin{proof}
Fix $\epsilon > 0$. 

We may assume that $r$, the degree of the alternating group $G$, is sufficiently large. For if $r \leq c$ with a universal constant $c$, then by choosing $\delta$ less than $1/c$ the condition of the lemma implies that $x=1$. The statement is clear for $x=1$. Let us assume that $x \not= 1$.  	
	
Let $\delta_0$ be such that $|x^{G}| = |G|^{\delta_0}$. We may assume that $\delta_{0} > 0$, for otherwise $x=1$. The upper bound of the proposition amounts to showing that $\delta_{0} \leq \Delta(x)(1+ \epsilon)$. For every $\epsilon_1 > 0$ there exists $N_1$ such that whenever $r \geq N_{1}$ then $\Delta(x) > (1 - \epsilon_{1}) \delta_{0}$ by Lemma \ref{GaronziMaroti}. Thus it suffices to choose $\epsilon_1$ such that $1 < (1 - \epsilon_{1})(1+ \epsilon)$. This is the case when $\epsilon_{1} < \epsilon / (1 + \epsilon)$. 

It remains to establish the lower bound of the proposition. We first prove the same statement for the symmetric group $H$ of degree $r$. For each integer $i$ with $1 \leq i \leq r$, let $c_{i}$ be the number of cycles of length $i$ in the disjoint cycle decomposition of $x$. We have 
\begin{equation}
\label{long}	
|C_{H}(x)| = \Big( \prod_{i=1}^{r} c_{i}! \Big) \Big( \prod_{i=1}^{r} i^{c_{i}} \Big) \leq \Big( \sum_{i=1}^{r} c_{i} \Big)! \Big( \prod_{i=2}^{r} i^{c_{i}} \Big) = t! \Big( \prod_{i=2}^{r} i^{c_{i}} \Big),
\end{equation}
where $t$ is the number of cycles in the disjoint cycle decomposition of $x$. Observe that $t = r (1 - \Delta(x))$. This and Lemma \ref{Stirling} give 
\begin{equation}
\label{longer}
\begin{split}	
t! \leq 2 \sqrt{2 \pi t} {\Big(\frac{t}{e}\Big)}^{t} \leq  2 \sqrt{2 \pi r} {\Big(\frac{r}{e}\Big)}^{t} = 2 \sqrt{2 \pi r} {\Big(\frac{r}{e}\Big)}^{r (1 - \Delta(x))} = \\
= 2 {\Big(\sqrt{2 \pi r}\Big)}^{\Delta(x)} {\Big(\sqrt{2\pi r} {\Big(\frac{r}{e}\Big)}^{r} \Big)}^{1 - \Delta(x)}  \leq 2 {\Big( \sqrt{2\pi r} \Big)}^{\Delta(x)} |H|^{1 - \Delta(x)}.
\end{split}
\end{equation}
We have 
\begin{equation}
\label{short}
2 {\Big( \sqrt{2\pi r} \Big)}^{\Delta(x)} \leq |H|^{(\epsilon/2)\Delta(x)}
\end{equation}
for every large enough $r$. By considering the derivative of the function $f(x) = x^{1/x}$, we see that $i^{1/i} \leq e^{1/e}$ for every positive integer $i$. It follows that 
\begin{equation}
\label{i}	
\prod_{i=1}^{r} i^{c_{i}} = \prod_{i=2}^{r} i^{(ic_{i})/i} \leq \prod_{i=2}^{r} e^{ic_{i}/e} = e^{(\sum_{i=2}^{r} i c_{i})/e}.	
\end{equation}
Now $\sum_{i=2}^{r} i c_{i} \leq \sum_{i=2}^{r} 2(i-1)c_{i} = 2 (\sum_{i=1}^{r}(i-1)c_{i}) = 2 \Delta(x) r$. Applying this to (\ref{i}) gives
\begin{equation}
\label{ii}
\prod_{i=1}^{r} i^{c_{i}} \leq e^{2 \Delta(x) r/e} < |H|^{(\epsilon/2) \Delta(x)},
\end{equation}
holding for every sufficiently large $r$. By (\ref{long}), (\ref{longer}), (\ref{short}), and (\ref{ii}), we obtain 
$$|C_{H}(x)| < |H|^{(\epsilon/2)\Delta(x)} \cdot |H|^{1 - \Delta(x)} \cdot |H|^{(\epsilon/2) \Delta(x)} = |H|^{1 - \Delta(x)(1-\epsilon)}.$$ Thus $|H|^{\Delta(x)(1-\epsilon)} < |x^{H}|$. This proves the claim for the symmetric group $H$. 

We proved above that for all $\epsilon_{1} > 0$ there exists $\delta_{1} > 0$ such that if $|x^{H}| \leq |H|^{\delta_1}$, then 
\begin{equation}
\label{star}	
|H|^{\Delta(x)(1-\epsilon_{1})} \leq |x^{H}|.
\end{equation}
We fixed $\epsilon > 0$. Take $\epsilon_{1} = \epsilon/2$ and $\delta < \delta_{1}/2$. Inequality (\ref{star}) gives $|x^{G}| > |H|^{\Delta(x)(1-(\epsilon/2))}/2$, which is at least $|G|^{\Delta(x)(1-\epsilon)}$ for every sufficiently large $r$, by noting that $\Delta(x) \geq 1/r$. This proves the lower bound of the proposition.   
\end{proof}

\section{Bounding conjugacy class sizes in simple classical groups}

The purpose of this section is to extend Proposition \ref{classsize1} for the case when $G$ is a simple classical group. We also record a consequence. 

Let $n \geq 2$ be an integer and $q$ a prime power. Let $G$ be one of the classical groups $\mathrm{L}_{n}^{\pm}(q)$, $\mathrm{PSp}_{n}(q)$ or $\mathrm{P\Omega}_{n}^{\pm}(q)$. Let $V = V_{n}(q^{u})$ be the natural module for the lift of $G$ where $u = 2$ if $G$ is unitary and $u=1$ otherwise. Let $\overline{\mathbb{F}}$ be the algebraic closure of $\mathbb{F}_{q}$ and let $\overline{V} = V \otimes \overline{\mathbb{F}}$. Let $x \in G$ and let $\hat{x}$ be a preimage of $x$ in $\mathrm{GL}(V)$. In \cite{LSS} the support $\nu(x)$ of $x$ is defined to be
$$\nu(x) = \nu_{V,\overline{\mathbb{F}}}(x) = \min \{ \dim [\overline{V}, \lambda \hat{x}] : \lambda \in \overline{\mathbb{F}}^{*} \}.$$ Define $a = a(G)$ to be $1$ if $G = \mathrm{L}_{n}^{\pm}(q)$ and $1/2$ otherwise.

The following is \cite[Proposition 3.4]{LSS}. 

\begin{proposition}
\label{LSS}	
For any $\epsilon > 0$, there exists $\delta > 0$ such that if $x$ is an element of a simple classical group $G$ with $|x^{G}| \leq |G|^{\delta}$, then 
$$q^{(2a-\epsilon)n \nu(x)} \leq |x^{G}| \leq q^{(2a + \epsilon)n \nu(x)}.$$
\end{proposition}  

For $x \in G$ where $G$ is a simple classical group, let $$\Delta(x) = \frac{\nu(x) \cdot 2a \cdot n \cdot \log q}{ \log |G|}.$$ We may now state the main result of this section. 

\begin{proposition}
	\label{classsize11}	
	For all $\epsilon > 0$ there exists $\delta > 0$ such that whenever $G$ is an alternating group or a simple classical group and $x \in G$ with $|x^{G}| \leq |G|^{\delta}$, then 
	$$|G|^{\Delta(x)(1- \epsilon)} \leq |x^{G}| \leq |G|^{\Delta(x)(1+ \epsilon)}.$$
\end{proposition}

\begin{proof}
Fix $\epsilon > 0$. We may assume that $G$ is a simple classical group with parameters $n$, $q$ and $a$, by Proposition \ref{classsize1}. Since $|G|^{\Delta(x)} = q^{2a n \nu(x)}$, the conclusion of the proposition is
\begin{equation}
\label{conclusion}	 
q^{2a n \nu(x)(1- \epsilon)} \leq |x^{G}| \leq q^{2a n \nu(x)(1+ \epsilon)}.
\end{equation}
Let $\epsilon_{1} > 0$ be such that $\epsilon_{1} < 2a\epsilon$. Choose $\delta > 0$ for $\epsilon_{1}$ such that Proposition \ref{LSS} is satisfied. Assume that $|x^{G}| \leq |G|^{\delta}$. Then (\ref{conclusion}) follows from Proposition \ref{LSS}.
\end{proof}

We will need the following technical consequence of Proposition \ref{classsize11}.

\begin{corollary}
\label{corollary}	
There exists $\delta > 0$ such that whenever $G$ is a (finite) alternating or simple classical group and $x_{1}, \ldots , x_{k} \in G$ such that $|x_{1}^{G}| \cdots |x_{k}^{G}| \leq |G|^{\delta}$, then there exists $z \in x_{1}^{G} \cdots x_{k}^{G}$ with $\Delta(z) = \Delta(x_{1}) + \ldots + \Delta(x_{k})$. 
\end{corollary}

\begin{proof}
Choose $\delta > 0$ such that whenever $G$ is an alternating group or a simple classical group and $x \in G$ with $|x^{G}| \leq |G|^{\delta}$, then $\Delta(x) < 1/4$. Such a $\delta$ exists by Proposition \ref{classsize11}. 	
	
Let $x_{1}, \ldots , x_{k}$ be elements in an alternating or simple classical group $G$ such that $|x_{1}^{G}| \cdots |x_{k}^{G}| \leq |G|^{\delta}$. For each $i$ with $1 \leq i \leq k$, let $s_{i} = \Delta(x_{i})$. Put $s = \sum_{i=1}^{k} s_{i}$.  	
	
For every $i$ with $1 \leq i \leq k$, the inequality $|x_{i}^{G}| \leq {|G|}^{\delta}$ implies that $s_{i} < 1/4$. Let $i$ and $j$ be two distinct indices from $\{ 1, \ldots , k \}$. We have $|x_{i}^{G} x_{j}^{G}| \leq |x_{i}^{G}||x_{j}^{G}| \leq {|G|}^{\delta}$, $s_{i} < 1/4$ and $s_{j} < 1/4$. Since both $s_i$ and $s_j$ are less than $1/4$, the normal set $x_{i}^{G}x_{j}^{G}$ contains a conjugacy class $y^{G}$ with $y \in G$ and $\Delta(y) = s_{i}+s_{j}$ by \cite[Lemma 3.5]{LSS}, for classical groups $G$. The same statement holds when $G$ is an alternating group. Since $|y^{G}| \leq {|G|}^{\delta}$, we have $s_{i} + s_{j} = \Delta(y) < 1/4$. Continuing in this way, we find that there is an element $z \in G$ such that $z^{G}$ is contained in $x_{1}^{G} \cdots x_{k}^{G}$, and $z$ satisfies $\Delta(z) = s_{1} + \cdots + s_{k} = s$ and $s$ is less than $1/4$. 	
\end{proof}

\section{Lower bounds on conjugacy class sizes in simple groups}

Let $G$ be a non-abelian finite simple group different from a sporadic group. We define the rank of $G$ to be its untwisted Lie rank if it is a group of Lie type and to be its degree if it is an alternating group (and not a group of Lie type). 

\begin{lemma}
	\label{classsize}
	Every non-trivial conjugacy class of a non-abelian finite simple group of rank $r$ has size at least $|G|^{1/16r}$.
\end{lemma}

\begin{proof}
	Let $G = G_{r}(q)$ be a finite simple group of Lie type of rank $r$ defined over $\mathbb{F}_{q}$, the finite field of order $q$. Let $x$ be an arbitrary non-trivial element in $G$. We have $$q^{r/2} \leq |x^{G}| \leq |G| \leq q^{8r^{2}}$$ by \cite[Proposition 2.2]{GPSz}. The result follows in this case. Let $G$ be the alternating group of degree $r \geq 5$. Since the minimal index of a proper subgroup of $G$ in $G$ is $r$, every non-trivial conjugacy class of $G$ has size at least $r > r^{1/16} \geq |G|^{1/16r}$.  
\end{proof}

The following is \cite[Theorem 2.2]{LSS}.

\begin{lemma}
	\label{normalset}
	For any $\epsilon > 0$ there exists $N$ such that if $G$ is a non-abelian finite simple group of rank at least $N$ and $B$ is a non-empty normal subset of $G$, then $B$ contains a conjugacy class of $G$ of size at least $|B|^{1 - \epsilon}$. 
\end{lemma}

We are in position to prove the following result. 

\begin{proposition}
\label{log}
For any $\epsilon > 0$ there exists $\delta > 0$ such that whenever $B_{1}, \ldots , B_{k}$ are non-empty normal subsets in a non-abelian finite simple group $G$ with $$|B_{1}| \cdots |B_{k}| \leq |G|^{\delta},$$ then there exists $z \in B_{1} \cdots B_{k}$ such that 
$$|z^{G}| \geq (|B_{1}| \cdots |B_{k}|)^{1-\epsilon}.$$  
\end{proposition}

\begin{proof}
Fix $\epsilon > 0$. We may assume that $\epsilon < 1$. Let $G$ be a non-abelian finite simple group. Let $k$ be a positive integer and let $B_{1}, \ldots , B_{k}$ be non-empty normal subsets in $G$. For each $i$ with $1 \leq i \leq k$, let $x_{i}$ be a member of a largest conjugacy class in $B_{i}$. We may assume that each $x_{i}$ is different from $1$. 

Assume first that $|G|$ is bounded from above by a constant $c$. If $\delta$ is chosen to be less than $1/c$, then $|G|^{\delta} < 2$, and the statement is clear. Thus from now on we may assume that $|G|$ is unbounded. In particular, we assume that $G = G_{r}(q)$ is a finite simple group of Lie type of rank $r$ defined over $\mathbb{F}_{q}$, the finite field of order $q$, or $G$ is the alternating group of degree $r \geq 5$.

Assume first that $r$ is bounded from above by a constant $c$. If $\delta$ is chosen to be less than $1/16c$, then the statement follows from Lemma \ref{classsize}. Thus from now on we may assume that $r$ is sufficiently large, that is, $G$ is a finite simple classical group whose lift acts naturally on a vector space of large enough dimension, or $G$ is the alternating group of large enough degree.   
 
We may assume by Lemma \ref{normalset} that for every $i$ with $1 \leq i \leq k$ we have $|x_{i}^{G}| \geq |B_{i}|^{1-\epsilon_1}$ for any fixed $\epsilon_{1} > 0$. If there exists $z \in x_{1}^{G} \cdots x_{k}^{G}$ such that 
\begin{equation}
\label{reduce}	
|z^{G}| \geq (|x_{1}^{G}| \cdots |x_{k}^{G}|)^{1-(\epsilon/2)}, 
\end{equation}
then $$|z^{G}| \geq (|B_{1}| \cdots |B_{k}|)^{(1 - \epsilon_{1})(1-(\epsilon/2))} \geq (|B_{1}| \cdots |B_{k}|)^{1-\epsilon}$$ whenever $\epsilon_{1}$ is chosen such that $\epsilon_{1} \leq \epsilon/(2 - \epsilon)$.   
 
In the rest of the proof we will find an element $z \in x_{1}^{G} \cdots x_{k}^{G}$ such that (\ref{reduce}) holds.  
 
%Let $n \geq 2$ be an integer and $q$ a prime power. Let $G$ be one of the classical groups $\mathrm{L}_{n}^{\pm}(q)$, %$\mathrm{PSp}_{n}(q)$ or $\mathrm{P\Omega}_{n}^{\pm}(q)$. Let $V = V_{n}(q^{u})$ be the natural module for the lift of %$G$ where $u = 2$ if $G$ is unitary and $u=1$ otherwise. Let $\overline{\mathbb{F}}$ be the algebraic closure of %$\mathbb{F}_{q}$ and let $\overline{V} = V \otimes \overline{\mathbb{F}}$. Let $x \in G$ and let $\hat{x}$ be a %preimage of $x$ in $\mathrm{GL}(V)$. In \cite{LSS} the support $\nu(x)$ of $x$ is defined to be
%$$\nu(x) = \nu_{V,\overline{\mathbb{F}}}(x) = \min \{ \dim [\overline{V}, \lambda \hat{x}] : \lambda \in %\overline{\mathbb{F}}^{*} \}.$$ In case $G$ is the alternating group of degree $r \geq 5$, then the role of $\nu(x)$ is %taken over by $\Delta(x)$ for $x \in G$ which is defined, as before, to be $(r-t)/r$ where $t$ is the number of cycles %in the disjoint cycle decomposition of $x$.  

We may assume that $|x_{1}^{G}| \cdots |x_{k}^{G}| \leq |G|^{\delta_{1}}$ where $\delta_{1}$ is a constant whose existence is assured by Corollary \ref{corollary}. Let $z \in x_{1}^{G} \cdots x_{k}^{G}$ such that $\Delta(z) = \sum_{i=1}^{k} \Delta(x_{i})$. For each $i$ with $1 \leq i \leq k$, let $s_{i} = \Delta(x_{i})$. Put $s = \sum_{i=1}^{k} s_{i}$. 

Let $\epsilon_{2} > 0$ be such that $\epsilon_{2} < \epsilon / (4 - \epsilon)$. Let $\delta_{2} > 0$ be a constant whose existence is assured by Proposition \ref{classsize11} for $\epsilon_{2}$. Let $\delta$ be the minimum of $\delta_{1}$ and $\delta_{2}$. On one hand Proposition \ref{classsize11} gives 
\begin{equation}
	\label{e22}
	|z^{G}| \geq |G|^{(1-\epsilon_{2})s}
\end{equation}
and on the other,
\begin{equation}
	\label{e33}
	|x_{1}^{G}| \cdots |x_{k}^{G}| \leq |G|^{(1 + \epsilon_{2})\sum_{i=1}^{k} s_{i}} = |G|^{(1 + \epsilon_{2})s}.
\end{equation}
Finally, inequality (\ref{reduce}) is satisfied since $(1-\epsilon_{2})s > (1 + \epsilon_{2})s(1 - (\epsilon/2))$.
\end{proof}

\section{Proof of Theorem \ref{main}}

Gill, Pyber, Short, Szab\'o \cite[Theorem 4.3]{GPSSz} proved the following important result.

\begin{proposition}
	\label{GPSSz}
	Let $A$ and $B$ be finite sets in a group $G$ with $B$ normal in $G$. Suppose that $|AB| \leq K |A|$ for some positive number $K$. Then there exists a nonempty subset $X$ of $A$ such that $|X B^{k}| \leq K^{k} |X|$ for $k \geq 1$. In particular, $|B^{2}| \leq K|B|$ implies that $|B^{k}| \leq K^{k}|B|$ for $k \geq 1$.
\end{proposition}

\begin{proof}[Proof of Theorem \ref{main}]
Fix $\epsilon > 0$. We may assume that $\epsilon < 1$. Choose $\delta_{1}$ satisfying the statement of Proposition \ref{log} with $\epsilon/2$. Let $\delta = (\epsilon/2) \cdot (1 + (\epsilon/2))^{-1} \delta_{1}$. Let $G$ be a non-abelian finite simple group. Let $B$ be a normal subset in $G$ and let $A$ be a subset of $G$, both of size at most $|G|^{\delta}$. The result is clear if $B=1$. Thus assume that $B \not= 1$. Let $k$ be the smallest positive integer for which $|A| \leq |B|^{(\epsilon/2)k}$. Then $|B|^{(\epsilon/2)(k-1)} \leq |A|$ and so 
\begin{equation}
\label{ee11}	
|B|^{(\epsilon/2)k} \leq |A| |B|^{\epsilon/2} \leq |G|^{\delta} |G|^{\delta (\epsilon/2)} = |G|^{(1+(\epsilon/2))\delta} = |G|^{(\epsilon/2)\delta_{1}}. 
\end{equation}

Let $K >0$ be the number defined by $|AB| = K |A|$. Let $X$ be a subset of $A$ whose existence is assured by Proposition \ref{GPSSz}. We get 
\begin{equation}
\label{eee111}	
|B^{k}| \leq |XB^{k}| \leq K^{k} |X| \leq K^{k} |A| \leq K^{k} |B|^{(\epsilon/2)k}
\end{equation}
by Proposition \ref{GPSSz}. We have 
\begin{equation}
\label{eeee1111}	
|B^{k}| \geq |B|^{(1-(\epsilon/2))k}
\end{equation}
by (\ref{ee11}) and Proposition \ref{log}. Inequalities (\ref{eee111}) and (\ref{eeee1111}) provide $|B|^{(1-(\epsilon/2))k} \leq K^{k}|B|^{(\epsilon/2)k}$, and so $K \geq |B|^{1 - \epsilon}$. The result follows.  
\end{proof}


\begin{thebibliography}{30}

\bibitem{Feller} W. Feller, An introduction to probability theory and its applications, Vol. 1, 3rd ed., Wiley, New York, 1968.

\bibitem{GaronziMaroti} M. Garonzi, A. Mar\'oti, Alternating groups as products of four conjugacy classes. \emph{Arch. Math. (Basel)} \textbf{116} (2021), no. 2, 121--130.

\bibitem{GPSz} N. Gill, L. Pyber, E. Szab\'o, A generalization of a theorem of Rodgers and Saxl for simple groups of bounded rank. \emph{Bull. Lond. Math. Soc.} \textbf{52} (2020), no. 3, 464--471.

\bibitem{GPSSz} N. Gill, L. Pyber, I. Short, E. Szab\'o, On the product decomposition conjecture for finite simple groups.
\emph{Groups Geom. Dyn.} \textbf{7} (2013), no. 4, 867--882. 

\bibitem{LNS} M. W. Liebeck, N. Nikolov, A. Shalev, Product decompositions in finite simple groups. \emph{Bull. Lond. Math. Soc.} \textbf{44} (2012), no. 3, 469--472.

\bibitem{LSS} M. W. Liebeck, G. Schul, A. Shalev, Rapid growth in finite simple groups. \emph{Trans. Amer. Math. Soc.} \textbf{369} (2017), no. 12, 8765--8779.

\bibitem{LS} M. W. Liebeck, A. Shalev, Diameters of finite simple groups: sharp bounds and applications. \emph{Ann. of Math. (2)}, \textbf{154} (2001), 383--406.

\bibitem{Shalev} A. Shalev, Word maps, conjugacy classes, and a noncommutative Waring-type theorem. \emph{Ann. of Math. (2)}, \textbf{170} (2009), 1383--1416.

\end{thebibliography}
\end{document}